\documentclass[a4paper,12pt,openany]{amsart}
\usepackage[T1]{fontenc}
\usepackage[utf8]{inputenc}
\usepackage {amsmath , amssymb, mathrsfs}
\usepackage {setspace}
\usepackage {amsthm}
\usepackage{graphicx, color}
\newtheorem {defn}{Definition}[section]

\newtheorem {teo}[defn]{Theorem}

\newtheorem{lemm}[defn]{Lemma}
\newtheorem{cor}[defn]{Corollary}

\newtheorem{ex}[defn]{Example}
\newtheorem{oss}[defn]{Remark}

\newcommand{\hada} {\star}

\usepackage[all]{xy}
\begin{document}
\title { On    Hadamard products   of linear varieties}
\thanks{Last updated: June 24, 2016}

\author{C. Bocci}
\address{Dipartimento di Ingegneria dell'Informazione e Scienze Matematiche,
Universit\`a di Siena\\
Via Roma, 56 Siena, Italy}
\email{cristiano.bocci@unisi.it}
\author{G. Calussi}
\address{Dipartimento di Matematica e Informatica "Ulisse Dini", Universit\`a di Firenze \\
Viale Morgagni 67/a, 50134 Firenze, Italy
\vskip0.1cm Dipartimento di Matematica e Informatica, Universit\`a di Perugia \\
via Vanvitelli 1, 06123 Perugia, Italy}
\email{gabriele.calussi@gmail.com}
\author{G. Fatabbi}
\address{Dipartimento di Matematica e Informatica, Universit\`a di Perugia \\
via Vanvitelli 1, 06123 Perugia, Italy}
\email{giuliana.fatabbi@unipg.it}
\author{A. Lorenzini}
\address{Dipartimento di Matematica e Informatica, Universit\`a di Perugia \\
via Vanvitelli 1, 06123 Perugia, Italy}
\email{annalor@dmi.unipg.it}

\begin{abstract} In this paper we address the Hadamard product of linear varieties not necessarily in general position.
\par
In $\mathbb{P}^2$ we obtain a complete description of the possible outcomes. In particular, in the case of two disjoint finite sets
$X$ and $X'$ of collinear points, we get conditions for $X\hada X'$ to be either a collinear finite set of points or a grid of
$| X|| X'|$ points.
\par
In $\mathbb{P}^3,$ under suitable conditions (which we prove to be generic), we show that $X\hada X'$ consists of $| X||X'|$ points on the two different rulings of a non-degenerate quadric and we compute its Hilbert function in the case $|X| = |X'|.$

\end{abstract}

\maketitle

\section{Introduction}
 The Hadamard product of matrices  is well known in linear algebra: it has nice properties in matrix analysis (\cite{HM,Liu,LT}) and has applications in both statistics and physics (\cite{Liu2, LN, LNP}).
Recently, in the papers \cite{CMS, CTY}, the authors define a Hadamard product between projective varieties $X,Y\subset {\mathbb{P}}^n$, denoted $X \hada Y$, as the  closure of the image of the rational map
\[
X \times Y \dashrightarrow  {\mathbb{P}}^n, \quad  ([a_0:\dots : a_n], [b_0: \dots :b_n])\mapsto [a_0b_0 :a_1b_1:\ldots :a_nb_n].
\]
and use it to describe the algebraic variety associated to the restricted Boltzmann machine, which is the undirected graphical model for binary random variables specified by the bipartite graph $K_{r,n}$ (note that  \cite{CTY} concerns the case $r=2, n=4$). Using  the definition in \cite{CMS}, the first author, together with E. Carlini and J. Kileel, started to study more deeply the Hadamard product of projective varieties, with particular emphasis for the case of linear spaces (\cite{BCK}).

Since the Hadamard product of varieties is far to be completely understood and studied, our paper wants to be a natural continuation of the paper \cite{BCK}. Here we still study the case of both linear spaces and zero-dimensional schemes in ${\mathbb{P}}^2$ and ${\mathbb{P}}^3$, dropping the hypothesis that these varieties  be in general position. The condition to be not in general position means that points can have many zero coordinates and linear spaces  can intersect coordinate hyperplanes in dimension greater than the expected one for the case of general position. This fact forces us to study all possible pathological behaviors  that then can happen in the Hadamard  product of such varieties. For the case of ${\mathbb{P}}^2$, Theorem \ref{teorettapunto} gives a complete classification of all possible cases of the Hadamard product between a point and a line, while Theorem \ref{retta2punti} studies all possible cases of incidence of $Q\hada L$ and $Q' \hada L$ for two distinct points $Q,Q'$ and a line $L$. These results lead to Theorem \ref{grid} where we prove that, under suitable conditions, the Hadamard product of  two sets of collinear points $X$ and $Y$ is a complete intersection. Recall that a set $X$ of points of ${\mathbb{P}}^n$ is a {\it complete intersection} if its ideal $I_X$ is generated by $n$ forms.
 \par
 Turning to the case of $\mathbb{P}^3$, we notice that the Hadamard product  of two sets of collinear points $X$ and $Y$ is not, in general, a complete intersection. However, we can prove, under generic assumptions, that $X \hada Y$ is a grid on a quadric (Theorem  \ref{teop3})  and we are able to compute its Hilbert function when $|X| = |Y|$  (Theorem  \ref{HF3}). In this case we also prove that $X\hada Y$ is never a complete intersection (assuming $|X| = |Y|>1$).
\par
We work over an algebraically closed  field $\mathbb{K}.$
\par
We denote by $HF_X$ the {\em Hilbert function} of a finite set of  projective points $X$, that is  $HF_X(t)=\dim_\mathbb{K} R/I_{X_t}$, where $R=\mathbb{K}[x_0, \dots, x_n]$ and $I_X$ is the (radical) ideal defining $X$. It is well known that, after being strictly increasing, $HF_X(t)=|X |$ for $t$ sufficiently large and we call {\em regularity index of $X,$} denoted by $\tau_{X},$ the least  degree where this happens:
$$\tau_{X}=\min\{t\geq 0\colon HF_X(t)=|X| \}. $$
\par
We wish to thank the referee for his/her useful comments and suggestions: in particular his/her questions inspired Remark \ref{unique} and Corollary \ref{ci}.

\section{General results in $\mathbb{P}^n$}

  As in \cite{BCK}, $H_i\subset \mathbb{P}^n$  denotes  the hyperplane defined by $x_i=0$ and
$$\Delta_i=\bigcup_{0\leq j_1<\dots <j_{n-i}\leq n}H_{j_1}\cap\cdots\cap H_{j_{n-i}}.$$

Recall that $\Delta_i$ can be viewed as the $i-$dimensional variety of points having at most $i+1$ non-zero coordinate, equivalently at least $n-i$ zero coordinates.

We set $\Delta_{-1}$ to be the set $\{(0,\dots,0)\}$ and we write  $P\hada Q\in \Delta_{-1}$ if it is not defined.

It easily follows from Lemma 3.2 of \cite{BCK} that:
\begin{teo} \label{teobocci}
$\quad$
\begin{enumerate}
\item Let $P,Q,A$ be points of  $\mathbb{P}^n$ with $A\not\in  \Delta_{n-1},$ then $P\hada A=Q\hada A$ if and only if $P=Q.$
\item Let $H\subset  \mathbb{P}^n$ be  a hyperplane defined by  $a_0x_0+\cdots +a_nx_n=0$ and  such that $H\cap \Delta_0=\emptyset,$  let $P,Q$  be points not in    $\Delta_{n-1}$ with $P=[p_0:\dots:p_n],$ then $P\hada H:\{\frac{a_0}{p_0}x_0+\dots+\frac{a_n}{p_n}x_n=0\}$  and  $P\hada H=Q\hada H$ if and only if  $P=Q.$
\item Let $P\not\in \Delta_{n-1}$, let  $H$, $K$ be two hyperplanes  such that $H\cap \Delta_0=\emptyset =K\cap \Delta_0,$ then $P\hada H=P\hada K$ if and only if $H=K$.
\end{enumerate}
\end{teo}

\begin{teo}
Let $P \in \mathbb{P}^n\setminus \Delta_{n-1}$, let $H$,$K$ be   two  hyperplanes  such that $H\cap \Delta_0=\emptyset= K\cap \Delta_0, $  then $P\hada (H\cap K)=( P\hada H )\cap (P\hada K)$.
\end{teo}
\begin{proof} We may assume $H\not= K,$ for otherwise the result is trivial.
\par For all $Q\in H\cap K,$ we have  $P\hada Q\in P\hada (H\cap K) $   and $P\hada Q\in (P\hada H)\cap (P\hada K),$ hence  $P\hada (H\cap K) \subseteq (P\hada H)\cap (P\hada K),$ being the right-hand  side a closed set. To see the other inclusion, by  $ (3)$ of  Theorem \ref{teobocci}, we have   $P\hada H\not=P\hada K,$ then $(P\hada H)\cap (P\hada K)$ is a linear subspace of dimension $n-2$. Since  $P\not\in\Delta_{n-1}, $ it follows from  Lemma  3.1 of \cite{BCK} that  $P\hada (H\cap K) $ is a linear subspace of dimension $n-2. $ Therefore $P\hada (H\cap K) $ is the intersection of the hyperplanes $P\hada H$ and $P\hada K$.
\end{proof}

\begin{cor}\label{hadarettap3}
Let $ P\in \mathbb{P}^3\setminus  \Delta_2$ and let  $H,K$  be two planes such that  $L =H\cap K$ and  $H\cap \Delta_0=\emptyset= K\cap \Delta_0, $    then $P\hada L$ is the intersection of the two  planes $P\hada H$ and $P\hada K$.
\end{cor}

Now we look at the products of two hyperplanes $H$ and $K .$
\begin{oss}\label{coord}{\em If $H$ and $K$ are coordinate hyperplanes respectively defined by $x_i=0$ and $x_j=0,$ then $H\hada K$ is the hyperplane defined by $x_i=0,$ when $i=j$ and the linear subspace defined by $x_i=x_j=0,$ when $i\neq j.$}
 \end{oss}

\begin{teo}\label{hyp_2_var}
Let $H,$  $K$ be  the hyperplanes of  $\mathbb{P}^n$  defined by $a_ix_i+a_jx_j=0$ and $b_ix_i+b_jx_j=0$ respectively, with $i\neq j$ in $\{0,\dots,n\}$ and either  $a_i a_j \neq 0$ or   $b_i b_j \neq 0.$ Then $H\star K$ is the hyperplane defined by $a_ib_ix_i-a_jb_jx_j=0.$
 \end{teo}

 \begin{proof}  For simplicity of notation we may assume $i=0$ and $j=1,$ the other case being similar.
 \par
 We distinguish the following two cases:
\begin{enumerate}
\item If $a_1=0$ and $b_0b_1\not=0$, then $H:\{x_0=0\}.$
Let  $P=[0:p_1:p_2:\cdots:p_n]\in H$ and let  $Q\in K$, then $Q=[-\frac{b_1}{b_0}q_1:q_1:q_2:\cdots:q_n],$ and so $P\hada
Q=[0:p_1q_1:\cdots:p_nq_n]$, i.e. $H\hada K: \{x_0=0\}$.
\item If $a_0a_1\not=0$ and $b_0b_1\not=0,$ then   $P=[-\frac{a_1}{a_0}p_1:p_1:p_2\cdots:p_n]\in H$ and
$Q=[-\frac{b_1}{b_0}q_1:q_1:q_2:\cdots:q_n]\in K,$ thus  $P\hada Q=[\frac{a_1b_1}{a_0b_0}p_1q_1:p_1q_1:p_2q_2:\cdots:p_nq_n]$. We claim  that $H\hada K$ is the hyperplane  $L :\{ a_0b_0x_0-a_1b_1x_1=0\}$.
It is obvious that $H\hada K\subseteq L $. To see the other inclusion let $S=[s_0:\cdots:s_n]\in L$ and  let $P=[-\frac{a_1}{a_0} :1:1:\dots:1]\in H$  and consider $Q=[-\frac{a_0s_0}{a_1}: s_1 : s_2 :\cdots: s_n ].$
Clearly $P\hada Q=S$  and $b_0\frac{-(a_0s_0)}{a_1}+b_1 s_1 =0,$ i.e. $Q\in K$.
\end{enumerate}
\end{proof}

\begin{cor}
Let $H\subset\mathbb{P}^n$ be  the hyperplane defined by   $a_ix_i+a_jx_j=0, $   with $i\neq j$ in  $\{ 0,\dots,n\}$ and $a_ia_j\neq 0,$  then $H\hada H $ is the hyperplane defined by  $a_i^2x_i-a_j^2x_j=0$.
\end{cor}
\begin{proof}
It follows immediately  from  Theorem \ref{hyp_2_var} with $H=K$.
\end{proof}

\begin{oss}\label{Hi}
{\em  If $Q\in H_i,$  for some $i\in\{0,\dots,n\}, $ then, for all $X\subseteq \mathbb{P}^n$,  we have $Q\hada X\subseteq H_i,$ hence $H_i\hada X\subseteq H_i.$}
\end{oss}

\begin{ex}\rm
Let $H$ and $K$ be the planes in ${\mathbb{P}}^3$ of equations respectively
\[
H: 3x_1-2x_3=0 \quad K: -7x_1+4x_3=0.
\]
Using the procedure, in \texttt{Singular}, described in Section \ref{singular} we get

{\small{\begin{verbatim}
ring r=0,(x(0..3)),dp;
ideal H=3*x(1)-2*x(3);
ideal K=-7*x(1)+4*x(3);
HPr(H,K,3);
_[1]=21*x(1)-8*x(3)
\end{verbatim}}}
\noindent and, in particular,
{\small{\begin{verbatim}
HPr(H,H,3);
_[1]=9*x(1)-4*x(3)
\end{verbatim}}}
\end{ex}

\section{Sets of collinear points in ${\mathbb{P}}^2$}
Now we focus on $n=2$ and on sets of at least two collinear points.

The following Corollary is an application  of the  results of the previous section.

\begin{cor}\label{corLhadaL2var}
Let $X,Y$ be two  sets of collinear  points in $\mathbb{P}^2$ such that $X\cup Y $ is contained in a line    $L$ such that $L\cap  \Delta_0\not =\emptyset, $  then $X\hada Y$ are collinear points  contained in the line $L\hada L$.
\end{cor}

We need the following technical Lemma, whose proof follows easily from the Definitions:
\begin{lemm}\label{rettadelta1}
Let $L$ be  a line in $\mathbb{P}^2$ such that $L\cap  \Delta_0=\emptyset. $  Then
 \begin{enumerate}

 \item $|L\cap\Delta_1|=3$ and $ |L\cap\Delta_1\cap H_i |=1,$ for all $i\in \{0,1,2\}.$
 \item For all $P,Q\in L\cap\Delta_1,$ we have that $P\hada Q\not\in \Delta_0$ if and only  if $P=Q.$
 \end{enumerate}
\end{lemm}

\begin{teo}
Let $X,Y\subseteq \mathbb{P}^2$ be two  sets  of   points with $|X|,|Y|\geq 3$,  $|X\cup Y|\geq 4$ and  $X\cup Y\subseteq L,$  where $L$ is a line.   Then the points of  $X\hada Y$  are not collinear if and only if $L\cap  \Delta_0=\emptyset. $
\end{teo}

\begin{proof}
 The necessary part follows from Corollary \ref{corLhadaL2var}.
\par To prove the sufficient part,  first we note that,  by  $(1)$ of Lemma \ref{rettadelta1},   $|\Delta_1\cap L|= 3,$   and so  $X\cup Y\not \subseteq \Delta_1,$  thus    there exists  at least  one point of $X$ or $Y$  not in  $\Delta_1$.
\par Suppose that  there exists a unique point   $P_1\in X$ such that $P_1\not\in\Delta_1,$ hence $X\setminus \{P_1\}\subseteq L\cap \Delta_1.$
 \par Then  either there exists $Q_1\in Y$ with $Q_1\not =P_1$ such that $Q_1\not\in\Delta_1$ or $Y\setminus \{P_1\}\subseteq L\cap \Delta_1.$
\par In the first case  for each $P\in X$ and for each $Q\in Y,$ with $P\not=P_1$ and $Q\not=Q_1,$  we have that  $P\hada Q_1\not=P_1\hada Q_1\not= P_1\hada Q,$ by $ (1)$ of Theorem \ref{teobocci}.
\par Since $P_1\hada L\not= Q_1\hada L$ by   $ (2)$ of Theorem \ref{teobocci}, the points $P\hada Q_1,P_1\hada Q_1, P_1\hada Q$  cannot be collinear.
\par In the second case we can  show that the points of $X\hada (Y\setminus \{P_1\})$ are not collinear. In fact, for every   $Q,Q'\in Y\setminus\{P_1\}$ with $Q\not= Q'$  and for each  $P\in X \setminus\{P_1\},$ we have $Q\hada L\subseteq H_i$ and $Q'\hada L\subseteq H_j,$ with $i\not=j,$ since  $L\cap  \Delta_0=\emptyset. $    Then $P_1\hada Q, P\hada Q \in H_i $ and $P_1\hada Q', P\hada Q'\in H_j. $
Now, we have either  $P\not =Q$ or $P\not= Q', $  whence, by  $ (2)$ of Lemma \ref{rettadelta1}, $P\hada Q\in \Delta_0$ or $P\hada Q'\in \Delta_0 .$ Thus, either $P_1\hada Q\not =P\hada Q\not=P_1\hada Q'$ or $P_1\hada Q\not =P\hada Q'\not=P_1\hada Q',$ for $P_1\hada Q,P_1\hada Q'\not\in \Delta_0.$  On the the other hand, $P_1\hada Q\not=P_1\hada Q'$  by $ (2)$ of Theorem  \ref{teobocci}. Since $P_1\hada Q,P_1\hada Q'\not\in H_i\cap H_j\subset \Delta_0,$ the points $P_1\hada Q, P\hada Q ,P_1\hada Q', P\hada Q' $  cannot be collinear.
\par Now suppose  that  there exist at least  $P_1,P_2\in X$ and $Q_1,Q_2\in Y$, all  distinct,   such that  either $P_1, P_2\not \in \Delta_{1}$ or $Q_1, Q_2\not \in \Delta_1.$
\par We  may assume  that $P_1, P_2\not \in \Delta_{1},$   the other case  being similar.
\par For every   $Q,Q'\in Y,$  with $Q\not=Q'$ we have that $P_1\hada Q\not=P_1\hada Q' $ and $P_2\hada Q\not= P_2\hada Q' $, by   $ (1)$ of Theorem \ref{teobocci}.  Since  $ P_1\hada L\not= P_2\hada L,$ by   $ (1)$ of Theorem \ref{teobocci}, the points $P_1\hada Q, P_1\hada Q' ,P_2\hada Q, P_2\hada Q' $ cannot be all collinear.
  \end{proof}

By Lemma 3.1 of  \cite{BCK}  we have that if $L\subset \mathbb{P}^n$ is a linear subspace of dimension $m$  and $P$ is a point, then $P\hada L$ is either empty or it is a linear subspace of dimension at most $m.$ If $P\not\in \Delta_{n-1}$ then $\dim( P\hada L)=m.$ In the following Theorem  we give a description of  what occurs  in  the plane in some cases.  We will use these technical results later.

\begin{teo}\label{teorettapunto}
Let $L$ be the line in $\mathbb{P}^2$ defined by $a_0x_0+a_1x_1+a_2x_2=0,$ set $A=[a_0:a_1:a_2]\in\mathbb{P}^2$ and let  $Q=[q_0:q_1:q_2]$ be any point.  Then:
\begin{enumerate}
\item if $Q\not\in\Delta_1$, then $Q\hada L$ is the line defined by $\frac{a_0}{q_0}x_0+\frac{a_1}{q_1}x_1+\frac{a_2}{q_2}x_2=0;$
 \item if $Q\in\Delta_1\setminus \Delta_0$ (and so   $Q\in H_j,$ for some $j\in\{ 0,1,2\}$),  $A\in \Delta_i\setminus \Delta_{i-1},$ $(i\in\{ 0,1,2\})$ and  $Q\hada A\in \Delta_{i-1},$  then $Q\hada L= H_j;$
    \item if $Q\in\Delta_1\setminus \Delta_0$ (and so   $Q \in H_j$  for some $j\in\{ 0,1,2\} $) and  we are not in the hypothesis of $(2)$,  then   $Q\hada L$ is the  point of  the intersection  of the line $H_j$ with the line defined by  $a_kq_lx_k+a_lq_kx_l=0,$ with $k,l\not=j;$
  \item if $Q\in\Delta_0$ and $Q\not=A,$ then $Q\hada L=Q$;
  \item if $Q\in\Delta_0$ and $Q=A$, then $Q\hada L$ is not defined.
\end{enumerate}
\end{teo}
\begin{proof}
\begin {enumerate}
\item [  $ (1)$]
If   $A\not\in\Delta_1$ the result  follows immediately from  $ (2)$ of Theorem \ref{teobocci}.
\par   Now we suppose  $A\in\Delta_1\setminus \Delta_0, $  say   $A\in H_2$, then $L:a_0x_0+a_1x_1=0$.
Let $P\in L,$ then $P=[-\frac{a_1}{a_0}p_1:p_1:p_2]$ and $P\hada Q=[-\frac{a_1}{a_0}q_0p_1:q_1p_1:q_1p_2],$ whence  $Q\hada L$ is contained in the line $a_0q_1x_0+a_1q_0x_1=0$. To see the other inclusion consider  $S=[-\frac{a_1q_0}{a_0q_1}s_1:s_1:s_2].$ Then we have that $S=Q\hada P,$ where
$P=[-\frac{a_1}{a_0q_1}s_1:\frac{s_1}{q_1}:\frac{s_2}{q_2}].$ Since  $a_0(-\frac{a_1}{a_0q_1}s_1)+a_1(\frac{s_1}{q_1})=0,$ then $P\in L.$
  \par   The proof is similar  if we suppose  $a_0=0$ or $a_1=0.$
\par   Finally suppose    $A\in\Delta_0, $ so that $L=H_i,$ for some $i\in\{ 0,1,2\}$. In this case it is easy  to see that $Q\hada L$ is $H_i.$
\item [ $ (2),(3)$]
  Since $Q\in\Delta_1\setminus \Delta_0,$ without loss of generality, we may assume that  $Q=[0:q_1:q_2] $. First  consider the case $i=2,$ i.e. $A\not\in \Delta_1.$ In this case necessarily $Q\hada A\in \Delta_1.$
    Let $P\in L$, then $P=[-\frac{a_1p_1+a_2p_2}{a_0}:p_1:p_2]$ and $P\hada Q=[0:p_1q_1:p_2q_2],$ i.e $Q\hada L=H_0,$ because $q_1q_2\not=0$.
\par   Now suppose  $i=1, $ i.e.  $ A\in\Delta_1\setminus \Delta_0.$ In this case   we need to distinguish
 whether  $Q\hada A\in \Delta_0$ or not.
\par   If $Q\hada A\in \Delta_0,$ then $a_0\not=0$ and  we may  assume  $a_2=0$, i.e.  $L:\{a_0x_0+a_1x_1=0\}.$
\par   Let $P\in L$, then $P=[-\frac{a_1p_1}{a_0}:p_1:p_2]$ and $P\hada Q=[0:p_1q_1:p_2q_2],$ i.e $Q\hada L = H_0 ,$ because $q_1q_2\not=0.$
\par   If $Q\hada A\in \Delta_1\setminus \Delta_0,$ then we may  assume $q_0=a_0=0,$ i.e.  $L:\{a_1x_1+a_2x_2=0\}.$
Let $P\in L,$ then $P=[p_0:-\frac{a_2}{a_1}p_2:p_2]$ and $Q\hada P=[0:-\frac{a_2}{a_1}p_2q_1:p_2q_2]=[0:-\frac{a_2}{a_1}q_1:q_2], $ whence $Q\hada L =[0:-\frac{a_2}{a_1}q_1:q_2]=H_0\cap \{a_1q_2x_1+a_2q_2x_2=0\}. $
\par   Now suppose $i=0,$ i.e.   $A\in \Delta_0,$ then $Q\hada A$  can be defined or not.
\par   If $Q\hada A \in \Delta_{-1},$   then    $L:\{x_0=0\}=H_0$  and  $Q\hada L=H_0.$
\par   If $Q\hada A\not\in \Delta_{-1},$  then we  may assume  $L:\{x_1=0\}=H_1. $  In this case $Q\hada L=\{[0:0:1]\}=H_0\cap \{x_1=0\}.$
\item [ $ (4)$, $ (5)$] They follow immediately from the definition of the Hadamard product.
\end{enumerate}
\end{proof}

\begin{cor}
Let $X$ be a set of collinear points in $\mathbb{P}^2$ and let $Q$ be  a point. Then $Q\hada X$  is  contained in a line.
\end{cor}

\begin{teo}\label{retta2punti}
Let $L$ be  a line in $\mathbb{P}^2$ defined by $a_0x_0+a_1x_1+a_2x_2=0,$ set  $A=[a_0:a_1:a_2]$  and let $Q=[q_0:q_1:q_2],Q'=[q_0':q_1':q_2']$ be two  distinct points. Suppose $A, Q, Q'\not\in\Delta_0,$ then:
\begin{enumerate}
\item $Q\hada L$ and $Q'\hada L$ are lines when:
\begin{enumerate}

\item $A\not\in\Delta_1$ and either  $Q\not\in\Delta_1$ or $Q'\not\in\Delta_1.$ The two  lines are distinct;
\item $A\not\in\Delta_1,$ $Q,Q'\in\Delta_1 .$  The two  lines are distinct if and only if $Q\hada Q'\in\Delta_0;$
\item $A\in\Delta_1 ,$ $Q,Q'\not\in\Delta_1.$ The two  lines are distinct  if and only if $det\left(\begin{array}{cc} q_j & q_k\\    q'_j & q'_k\end{array}\right)\not=0$ where $j,k\not=i$ and $A\in H_i;$
\item $A\in\Delta_1 ,$   $Q'\not\in\Delta_1$ and $Q\hada A\in\Delta_0.$  The two  lines are distinct;
\item $A \in\Delta_1, $ $Q\hada A\in\Delta_0$ and $Q'\hada A\in\Delta_0.$  The two  lines are distinct if and only if $Q\hada Q'\in\Delta_0$.

\end{enumerate}
\item $Q\hada L$ is a point and $Q'\hada L$ is a line when:
\begin{enumerate}
\item $A, Q\in\Delta_1 ,$   $Q'\not\in\Delta_1$ and $Q\hada A\not\in\Delta_0.$ The point    $Q\hada L$   belongs to the line $Q'\hada L$     if and only if $det\left(\begin{array}{cc} q_j & q_k\\    q'_j & q'_k\end{array}\right)=0,$ where $j,k\not=i$ and  $A\in H_i;$
\item $A, Q\in\Delta_1 $, $Q\hada A\not\in\Delta_0$ and $Q'\hada A\in\Delta_0.$ The point   $Q\hada L$   does not belong to the line $Q'\hada L.$
\end{enumerate}
\item $Q\hada L$ and $Q'\hada L$ are two distinct points when $A, Q,Q'\in\Delta_1, $  $Q\hada A\not\in\Delta_0 \not\ni Q'\hada A. $
\end{enumerate}
\end{teo}
\begin{proof}
\begin{enumerate}
\item [  $ (1)$-(a)]
If  $Q,Q'\not\in\Delta_1$ then $Q\hada L$ and $Q'\hada L$ are two  distinct lines by  $ (2)$ of Theorem \ref{teobocci}.
\par   If $Q\in\Delta_1 $ and $Q'\not\in\Delta_1,$  then, by  $ (2)$ of Theorem \ref{teorettapunto}, $Q\hada L=H_i,$ for some $i\in\{ 0,1,2\},$  which is a line   distinct from $Q'\hada L$.
\item [  $ (1)$-(b)] Since  $Q,Q'\in\Delta_1\setminus \Delta_0$, then,  by  $ (2)$ of Theorem \ref{teorettapunto}, $Q\hada L=H_i$ and $Q'\hada L=H_j,$ for some $i,j\in\{ 0,1,2\}.$  Clearly, these lines are distinct if and only if $Q\hada
Q'\in\Delta_0$.
\end{enumerate}
\bigskip
In the rest of the proof we have $A\in\Delta_1\setminus \Delta_0$  and so,  without loss of generality, we may assume  $A=[0:a_1:a_2]\in H_0,$ with $a_1a_2\not=0,$   whence $L:\{a_1x_1+a_2x_2=0\}$.
\begin{enumerate}
\item [  $ (1)$-(c)]
Since  $Q,Q'\not\in\Delta_1,$ then, by   $ (1)$ of Theorem \ref{teorettapunto}, $Q\hada L:\{a_1q_2x_1+a_2q_1x_2=0\}$ and $Q'\hada  L :\{a_1q'_2x_1+a_2q'_1x_2=0\}$ are  lines.
Clearly, these lines are distinct if and only if $det\left(\begin{array}{cc} q_1 & q_2\\    q'_1 & q'_2\end{array}\right)\not=0$.
\item [  $ (1)$-(d)] Since $Q'\not\in\Delta_1,$ then $Q'\hada L:\{a_1q'_2x_1+a_2q'_1x_2=0\}, $ by   $ (1)$ of Theorem \ref{teorettapunto}.
  Since $Q\hada A\in\Delta_0,$ then  necessarily $Q\in\Delta_1, $  and so, by  $ (2)$ of Theorem \ref{teorettapunto}, $Q\hada L=H_i,$  for some $i\in\{ 1,2\}.$  These two lines are distinct because $Q'\not\in \Delta_1.$

\item [  $ (1)$-(e)] Since $Q\hada A, Q'\hada A\in\Delta_0, $ then  necessarily $Q, Q'\in\Delta_1, $ and so by   $ (2)$ of Theorem \ref{teorettapunto},  $Q\hada L=H_i$ and $Q'\hada L=H_j, $  for some $i\in\{ 1,2\}.$   Clearly, these lines  are distinct lines if and only if $Q\hada Q'\in\Delta_0$.

\item [ $ (2)$-(a)] Since $Q\hada A\not\in\Delta_0,$ then, by  $ (3)$  of Theorem \ref{teorettapunto}, $Q\hada L =[0:-a_2q_1:a_1q_2],$  which  belongs to the line $Q'\hada L: \{a_1q'_2x_1+a_2q'_1x_2=0\}$ if and only if
$det\left(\begin{array}{cc} q_1 & q_2\\    q'_1 & q'_2\end{array}\right)=0$.

\item [ $ (2)$-(b)]
  Since $Q\hada A\not\in\Delta_0$ and $Q'\hada A\in\Delta_0,$ then, by  $ (2)$, $ (3)$  of Theorem \ref{teorettapunto}, $Q\hada L= [0:-a_2q_1:a_1q_2]$ and $Q'\hada L=H_i$ for some $i\in\{ 1,2\}.$ Clearly, the point does not belong to the line.
\item [ $ (3)$] Since $Q\hada A, Q'\hada A\not\in\Delta_0$ and  $A, Q,Q'\in\Delta_1, $  then, by  $ (3)$  of Theorem \ref{teorettapunto}, $Q\hada L= [ 0:-a_2q_1:a_1q_2],$  and $Q' \hada L =[0:-a_2q'_1:a_1q'_2].$ These two points are distinct    because $Q\not=Q'$  and $ Q,Q'\in\Delta_1 $ imply $det\left(\begin{array}{cc} q_1 & q_2\\    q'_1 & q'_2\end{array}\right)\not=0.$
\end{enumerate}
\end{proof}

\begin{oss}\label{minors}
{\em Let $L$ be  a  line in $\mathbb{P}^2$  such that $L\cap \Delta_0=\emptyset ,$  and let  $P,Q\in L  ,$ with $P\not=Q.$ Then the minors of order two
  of the matrix $\left(\begin{array}{c} P \\   Q\end{array}\right)$ are all not zero.}
\end{oss}

\begin{teo}\label{rettanpunti}
Let $L$ be  a line in $\mathbb{P}^2$  defined by  $a_0x_0+a_1x_1+a_2x_2=0$  and let  $A=[a_0:a_1:a_2].$ Let $L'$ be a line in $\mathbb{P}^2$ defined by  $a'_0x_0+a'_1x_1+a'_2x_2=0$ and such that $ [a'_0:a'_1:a'_2]\not\in\Delta_1.$ Let  $X'\subseteq L'$ be  a set of $r$ collinear points and suppose  $L'\cap \Delta_1\subseteq X'.$    Then:
\begin{enumerate}
\item If $A\not\in\Delta_1$, then $X'\hada L$ is  a set of $r$ distinct lines.
\item If $A\in\Delta_1\setminus \Delta_0,$ then $X'\hada L$ is  a set  of $r-1$  distinct lines and a point which does not belong to any line of $X'\hada L.$
     \item If $A\in\Delta_0,$ then $X'\hada L=L.$
\end{enumerate}
\end{teo}
\begin{proof}
\begin{enumerate}
\item [  $ (1)$]  Since  $A\not\in\Delta_1,$  then    $\{P'\hada L \vert P'\in X'\setminus \Delta_1\}$ is a set of  $r-3$ distinct lines. In fact,  by   $ (1)$ of Lemma  \ref{rettadelta1}, $|L\cap\Delta_1|=3,$ and by  $ (2)$ of Theorem \ref{teobocci}, each $P'\hada L$ is a line and these lines are all distinct.
\par   By $ (2)$ of  Theorem \ref{teorettapunto}, if $P'\in X'\cap \Delta_1,$ then $P'\hada L= H_i$ (for some $i\in \{0,1,2\}$) and they are  distinct by   $ (1)$-(b) of Theorem \ref{retta2punti};
\item  [ $ (2)$] Since  $A\in\Delta_1\setminus \Delta_0, $  then    $\{P'\hada L \vert P'\in X'\setminus\Delta_1\}$ is a set of  $r-3$ distinct lines. In fact, by   $ (1)$ of Theorem \ref{teorettapunto},  each $P'\hada L$ is a line and it is easy to prove  that  they are all distinct by using Remark \ref{minors}.
\par   Since $L'\cap \Delta_1\subseteq X',$ by   $ (1)$ of Lemma  \ref{rettadelta1} there  exists only one point $P_1'\in X'\cap \Delta_1$ such  that $P_1'\hada A\not\in\Delta_0.$  Thus, by  $ (3)$ of Theorem \ref{teorettapunto},  $P_1'\hada L$ is a point which does not belong to any of the lines of  $\{P'\hada L \vert P'\in X'\setminus \Delta_1\}$  by  $ (2)$-(a) of Theorem \ref{retta2punti} in view of Remark \ref{minors}. For  the remaining two points $P_2',P_3'\in X'\setminus \Delta_1$  we have $P_2'\hada L=H_i$ and  $P_3'\hada L =H_j,$ for some $i,j\in \{0,1,2\},$  with $i\not=j.$
\item  [ $ (3)$]Since $A\in\Delta_0, $ then $L=H_i,$  for some $i\in \{0,1,2\}.$ By Remark \ref{Hi}  $X'\hada L\subseteq L. $ On the other hand, from   $L'\cap \Delta_1\subseteq X' , $    $ [a'_0:a'_1:a'_2]\not\in\Delta_1 $ and    $ (1)$ of Lemma  \ref{rettadelta1}, it follows that  there exists a point $P'\in( X'\cap \Delta_1)\setminus \Delta_0,$ such that $P'\hada A\in \Delta_{-1}.$ The conclusion follows from  $ (2)$ of Theorem \ref{teorettapunto}.
 \end{enumerate}
\end{proof}

 \begin{cor}\label{distinctlines}
Let $L$ be  a line in $\mathbb{P}^2$  defined by  $a_0x_0+a_1x_1+a_2x_2=0$  and let  $A=[a_0:a_1:a_2].$ Let $L'$ be a line in $\mathbb{P}^2$ defined by  $a'_0x_0+a'_1x_1+a'_2x_2=0$ and such that $ [a'_0:a'_1:a'_2]\not\in\Delta_1.$ Let  $X'\subseteq L'$ be  a set of $r$ collinear points.     Then:
\begin{enumerate}
\item If $A\not\in\Delta_1$, then $X'\hada L$ is  a set of $r$ distinct lines;
\item If $A\in\Delta_1\setminus \Delta_0,$ then $X'\hada L$ is either a set of $r$ distinct lines or  a set  of $r-1$  distinct lines and a point which does not belong to any line of $X'\hada L;$
     \item If $A\in\Delta_0$ and $r\geq 3,$ then $X'\hada L=L.$
\end{enumerate}
\end{cor}
\begin{proof}
The only difference with the previous Theorem  is that we no longer have  the hypothesis $L'\cap \Delta_1\subseteq X',$ and so the existence of $P'\in X'\cap\Delta_1$ such that $P'\hada A\not\in\Delta_0$  is not granted, therefore we can obtain $r$ lines and no extra point.
\par   As for  $ (3)$, if there exists $P'\in X'\setminus \Delta_1,$ we are done by   $ (1)$ of Theorem \ref{teorettapunto}. If every $P'\in X'$ is in $\Delta_1$, then, in view of   $ (1)$ of Lemma  \ref{rettadelta1}, we have $r=3$ and  $L'\cap \Delta_1= X'.$
\end{proof}

\begin{ex}\rm
Let $L'\subset {\mathbb{P}}^2$ be the line of equation $2x_0-3x_1+132x_2$ and let $X'\subset L'$ be the following set of five points (randomly chosen in $L'$ by \texttt{Singular})
{\small{\[
X'=\left\{ [27:238:5],[12:96:2],[15:142:3],[21:234:5],[33:242:5]\right\}
\]}}
 After setting $X'=Y,$  we get that the ideal $I$ of $Y$ is generated by I[1] and I[2], where:
{\small{\begin{verbatim}
I[1]=2*x(0)-3*x(1)+132*x(2)
I[2]=375*x(1)^5-89300*x(1)^4*x(2)+8505840*x(1)^3*x(2)^2+
       -405077872*x(1)^2*x(2)^3+9645291984*x(1)*x(2)^4-
        -91862394624*x(2)^5
\end{verbatim}}}

As $L$ consider the line $2x_0-3x_1-11x_2$; clearly we are in the case $A\not\in\Delta_1$.
Computing the Hadamard product $L\hada X'$, in \texttt{Singular} we get
{\small{\begin{verbatim}
ideal J=2*x(0)-3*x(1)-11*x(2);
ideal YL=HPr(I,J,2);
degree(YL);
// dimension (proj.)  = 1
// degree (proj.)   = 5
genus(YL);
-4
\end{verbatim}}}
\noindent which tell us that $L\hada X'$ is the union of five lines. In particular, looking at the primary decomposition of the ideal \texttt{YL} we recover the five lines
\[
\begin{array}{l}
16x_0-3x_1-528x_2=0\\
284x_0-45x_1-7810x_2=0\\
2380x_0-405x_1-70686x_2=0\\
260x_0-35x_1-6006x_2=0\\
220x_0-45x_1-7986x_2=0
\end{array}
\]

\end{ex}

\begin{lemm}\label{lemmpuntiugualihada}
Let   $L$ be the line  defined by  $a_0x_0+a_1x_1+a_2x_2=0$ and $L'$ defined by $a'_0x_0+a'_1x_1+a'_2x_2=0,$ let $A=[a_0:a_1:a_2]$ and $A'=[a'_0:a'_1:a'_2]$ with $A,A'\not\in\Delta_1.$ Let  $P_1,P_2\in L\setminus \Delta_1$ and $P_1', P_2'\in L'\setminus
\Delta_1$  with $\{P_1,P_2\}\cap \{P_1', P_2'\}=\emptyset.$  If $P_1\hada P_1'=P_2\hada P_2',$ then either $P_1=P_2$ and $P_1'=P_2'$ or $P_i\hada A=P_j'\hada A'$ for  $i,j\in\{1,2\}$  with $i\not=j.$
\end{lemm}
\begin{proof}
If $P_1=P_2, $ then because $P_1,P_2\not\in\Delta_1$ and $P_1\hada P_1'=P_2\hada P_2'$, then $P_1'=P_2',$ by   $ (1)$ of Theorem  \ref{teobocci}.
\par   Suppose $P_1=[p_{10}:p_{11}:p_{12}]\not=P_2=[p_{20}:p_{21}:p_{22}]$ and $P_1'=[p_{10}':p_{11}':p_{12}']\not=P_2'=[p_{20}':p_{21}':p_{22}'].$  Since   $P_1'\not=P_2', $  we have $P_1\hada P_2'\not=P_1\hada P_1'=P_2\hada P_2'. $ Through  $P_1\hada P_2'$ and $P_1\hada P_1'$  there is
only   the line $P_1\hada L'$  defined by $\frac{a'_0}{p_{10}}x_0+\frac{a'_1}{p_{11}}x_1+\frac{a'_2}{p_{12}}x_2=0$, and through  $P_1\hada P_2'$ and
$P_2\hada P_2'$  there is only  the line $P_2'\hada L$ defined by $\frac{a_0}{p_{20}'}x_0+\frac{a_1}{p_{21}'}x_1+\frac{a_2}{p_{22}'}x_2=0.$ Since  $P_1\hada P_1'=P_2\hada P_2',$ these two lines must coincide,  i.e.
$\frac{a'_i}{p_{1i}}=\alpha\frac{a_i}{p_{2i}'},$ for $i\in\{0,1,2\}, $  which gives  $P_1\hada A=P_2'\hada A'$.
\end{proof}

\begin{teo}\label{grid}
Let $X$   be a set of $n$  collinear points, with  $X\subseteq L\setminus \Delta_1 $ and $L$ defined by $a_0x_0+a_1x_1+a_2x_2=0.$ Let $X'$   be a set of $m$  collinear points, with  $X'\subseteq L'\setminus \Delta_1 $ and $L'$ defined by $a'_0x_0+a'_1x_1+a'_2x_2=0.$ Set  $A=[a_0:a_1:a_2]$ and $A'=[a'_0:a'_1:a'_2].$
If  $A,A'\not\in\Delta_1$ and $X\cap X'=\emptyset,$ then, $X\hada X'$ is the $nm$ element grid $(X\hada
L')\cap(X'\hada L)$  if and only if $P\hada A\not=P'\hada A',$ for all $P\in X$ and all  $P'\in X'.$
\end{teo}
\begin{proof}
 If $(X\hada L')\cap(X'\hada L)$ is a grid with $nm$ elements, then the lines $\{P\hada L'\mid P\in X\}$ and $\{P'\hada L\mid P'\in X'\}$ are all distinct.   With the same reasoning of  Lemma
\ref{lemmpuntiugualihada}, we can prove that   $P\hada A\not=P'\hada A'$ for all $P\in X$ and $P'\in X'.$
\par   Conversely, since $P\hada A\not=P'\hada A', $ then $\{P\hada L'\mid P\in X\}$ and $\{P'\hada L\mid P'\in X'\}$  are two families of distinct lines by Corollary \ref{distinctlines}. Moreover, since $P\hada A\not=P'\hada A'$, for all $P\in X$ and $P'\in X', $  as in the proof of Lemma \ref{lemmpuntiugualihada},  we obtain  that also $\{P\hada L', P'\hada L\mid P\in X, P'\in X'\}$  is a family of distinct lines. On the other hand it is easy to check that $X\hada X'=(X\hada L')\cap(X'\hada
L).$ Now suppose   $X\hada X'$  has fewer than  $nm$ elements, then there exist $P_1,P_2\in X$ and $P'_1,P'_2\in X' $  with $P_1\not=P_2$ and $P'_1\not=P'_2 $ such that   $P_1\hada P'_1=P_2\hada P'_2.$  By Lemma \ref{lemmpuntiugualihada} this forces
   $P_1\hada A=P'_2\hada  A'$ against the hypothesis.
   \end{proof}

 \begin{cor}\label{sameline}
 Let $X,Y$ be two disjoint  sets of points  both contained in the same line $L. $ Suppose   $X\cap\Delta_1=\emptyset=Y\cap \Delta_1$ and   $L\cap \Delta_0=\emptyset.$ If  $ |X|=n$ and   $ |Y|=m,$ then  $X\hada Y$ is the   $nm$ element grid $(X\hada L)\cap(Y\hada L).$
\end{cor}
 \begin{proof}
 Let $L$ be defined by $a_0x_0+a_1x_1+a_2x_2=0$ and let $A=[a_0:a_1:a_2],$ then, for all $P\in X$ and all $P'\in Y, $ we have  $P\hada A\not= P'\hada A$ by   $ (1)$ of Theorem \ref{teobocci}. Now the conclusion follows from Theorem \ref{grid}.
 \end{proof}

 \begin{cor}\label{generic}
Let $L, L'$ be two generic distinct lines in $\mathbb{P}^2.$ There is  a generic choice of a finite set of points  $X\subseteq L$   for which it is possible a generic choice of a  finite set of points $X'\subseteq L'$   such  that  $X\hada X'$ is the grid $(X\hada L')\cap(X'\hada L).$
\end{cor}
\begin{proof}
Let $L,L' $ be  defined by  $a_0x_0+a_1x_1+a_2x_2=0$ and $a'_0x_0+a'_1x_1+a'_2x_2=0$ respectively and let $A=[a_0:a_1:a_2]$ and $A'=[a'_0:a'_1:a'_2].$ We may assume  $L\cap\Delta_0=\emptyset=L'\cap\Delta_0,$ whence $A,A'\not\in \Delta_1.$  Let  $P,Q\in L\setminus \Delta_1$ and $P' \in L'\setminus \Delta_1$ be distinct points. By   $ (1)$ of Theorem \ref{teobocci} $P\hada A\not=Q\hada A,$ then either $P\hada A\not=P'\hada A'$ or $Q\hada A\not=P'\hada A'.$ Suppose $P\hada A\not=P'\hada A'$ and consider   the  $2\times 3$ matrix $M(\lambda,\mu)=\left(\begin{array}{c}A\\  A' \end{array}\right)\hada\left(\begin{array}{c}\lambda P+\mu Q\\  P' \end{array}\right)$ with $[\lambda:\mu] \in \mathbb{P}^1.$ Then   $M(1,0)$ has a non-zero $2\times 2$ minor. The corresponding minor in $M(\lambda,\mu)$ is a non-zero linear form  $F(\lambda,\mu).$ Let $P_0\in L $ be the point corresponding to the zero locus of $F.$  Thus the set $L\setminus\{  P _0\} $ is a non empty open subset $U$  of $L $. Moreover, if    $R\in U$ then    $R\hada  A\neq P'\hada A'.$ Now consider  a finite set of points  $X\subseteq  U\cap (L\setminus\Delta_1).$ For any  point $R\in X,$ by the same reasoning as before, we find a non empty open subset $U'_R$  of $L'$ such that $R\hada A\neq R'\hada A'$  for any point $R'\in U'_R.$  Set $U'=\cap_{R\in X} U'_R.$ If $X'\subseteq U'\cap( L'\setminus\Delta_1)$  is a finite set of points
  then $P\hada A\not=P'\hada A'$  for   all $P\in X$ and $P'\in X'.$
  \par Now the claim follows from Theorem \ref{grid}.
\end{proof}

\begin{oss}\label{HF2}\rm
The grids obtained in Corollaries   \ref{sameline} and  \ref{generic} are  complete intersections so the Hilbert Functions and even the resolutions are known.
\end{oss}

\begin{ex} \rm
Let $L$ and $L'$ be respectively the  lines $3x_0+x_1-30x_2=0$ and $67x_0-6x_1-110x_2$ (randomly chosen by \texttt{Singular}). Consider the sets of points (still randomly chosen by \texttt{Singular}), which satisfy the hypotheses of Theorem \ref{grid},
\[
\begin{array}{l}
X=\{[6:12:1], [22:54:4],[29:63:5]\}\subset L\\
X'=Y=\{[22:154:5],[28:221:5],[34:288:5], [18, 146,3]\}\subset L'
\end{array}
\]
Using the procedure of Section \ref{singular} we compute the ideal $I$ of $X\hada X'$ and then its Hilbert function
\begin{verbatim}
ideal I=HPr(X,Y,2);

HF(2,I,0)=1;
HF(2,I,1)=3;
HF(2,I,2)=6;
HF(2,I,3)=9;
HF(2,I,4)=11;
HF(2,I,5)=12;
HF(2,I,6)=12;
\end{verbatim}
that is $H(t)=12$ for $t\geq 5$.
\par
As expected, $X\hada X'$ is a complete intersection.
\end{ex}

\section{Sets of collinear points in ${\mathbb{P}}^3$}
We keep assuming that the sets of points under consideration have cardinalities at least two.

  \begin{lemm}\label{genericH}
  Let $L $ be a line in $\mathbb{P}^3$ such that $L\cap\Delta_0=\emptyset$  and let $H$ be  a generic  plane  trough $L.$ Then $H\cap\Delta_0=\emptyset.$ Equivalently,   if $A$ is the point  corresponding  to $H $ in the dual space,    then $A\not\in \Delta_2.$
\end{lemm}
 \begin{proof}
 It is immediate since the planes through $L$ which contain some coordinate point are finite.
 \end{proof}

\begin{teo}\label{teop3}
Let $L, L'$ be two lines in $\mathbb{P}^3$,  $L=H\cap K$, $L'=H' \cap K'$, let  $A$, $B$, $A'$ and $B'$  be the points which correspond to $H,K,H',K'$ in the dual space  and suppose  $A, B, A', B'\not\in\Delta_2$. Let $X\subseteq L$ and $X'\subseteq L'$ be two finite sets of points such that $X\cap\Delta_2=\emptyset= X'\cap\Delta_2$ and $X\cap X'=\emptyset.$  Suppose $rank\left(\begin{array}{c}A\\  B\\  A'\\  B'\end{array}\right)\hada\left(\begin{array}{c}P\\  P\\  P'\\  P'\end{array}\right)>2$ for all $P\in X\subset L$ and $P'\in X'\subset L',$ then,  $X\hada X'=(X\hada L')\cap(X'\hada L)$ and $|X\hada X'|=|X||X'|.$
\end{teo}
\begin{proof} Let $P=[p_0:\dots:p_3]\in X$ and $P'=[p_0':\dots:p_3']\in X',$ first we show that   $P\hada L'$ and $P'\hada L$  are distinct lines. They are lines by  Lemma 3.1 of  \cite{BCK}.  On the other hand, by Corollary \ref{hadarettap3}, we have that $P\hada L'=(P\hada H')\cap( P\hada K')$, and $P'\hada L=(P'\hada H)\cap (P'\hada K).$ If we had $P\hada L'=P'\hada L,$ then,  after denoting   $\frac{1}{P}=[\frac{1}{p_0}:\dots:\frac{1}{p_3}]$ and  $\frac{1}{P'}=[\frac{1}{p_0'}:\dots:\frac{1}{p_3'}],$ we  would have that $rank\left(\begin{array}{c}A\\   B\\   A'\\   B'\end{array}\right)\hada\left(\begin{array}{c}\frac{1}{P'}\\   \frac{1}{P'
}\\   \frac{1}{P}\\   \frac{1}{P}\end{array}\right)=2.$ But  a straightforward computation shows that $$rank\left(\begin{array}{c}A\\   B\\   A'\\   B'\end{array}\right)\hada\left(\begin{array}{c}\frac{1}{P'}\\    \frac{1}{P'}\\    \frac{1}{P}\\    \frac{1}{P}\end{array}\right)=rank\left(\begin{array}{c}A\\   B\\  A'\\   B'\end{array}\right)\hada\left(\begin{array}{c}P\\   P\\   P'\\  P'\end{array}\right),$$
in contradiction with the hypothesis.
\par   Now let  $P_1,P_2\in X, $ we shall show that if $P_1\hada L'=P_2\hada L',$ then $P_1=P_2$.
 In fact, let $P'\in X',$ then,  we just showed that  $P_1\hada L'$    and  $P'\hada L$ are distinct lines, hence  $ (P_1\hada L')\cap (P'\hada L)$  is the point   $P_1\hada P'.$ Similarly,  $P_2\hada P'=  (P_2\hada L')\cap( P'\hada L).$ Therefore  $P_1\hada P'=P_2\hada P'$, hence $P_1=P_2,$ by  $ (1)$ of Theorem \ref{teobocci}.
\par   In a similar way we can  prove  that, for any $P_1',P_2'\in X',$ if $P_1'\hada L=P_2'\hada L,$ then $P_1'=P_2'$.
Finally,  we prove that for any  $P_1,P_2\in X$ and for any  $P_1',P_2'\in X',$  $P_1\hada P_1'\not=P_2\hada P_2'$ provided $P_1\not= P_2$ and $P_1'\not=P_2'.$
  Assume, by contradiction, that  $P_1\hada P_1'=P_2\hada P_2'$ but then we would have $P_1\hada L'=P_2'\hada L.$

\end{proof}

\begin{oss}\rm
Observe that, under all hypotheses of Theorem \ref{teop3}, the hypothesis $rank\left(\begin{array}{c}A\\  B\\  A'\\  B'\end{array}\right)\hada\left(\begin{array}{c}P\\  P\\  P'\\  P'\end{array}\right)>2$ forces $rank\left(\begin{array}{c}A\\  B\\  A'\\  B'\end{array}\right)\hada\left(\begin{array}{c}P\\  P\\  P'\\  P'\end{array}\right)=3,$ since $rank\left(\begin{array}{c}A\\  B\\  A'\\  B'\end{array}\right)\hada\left(\begin{array}{c}P\\  P\\  P'\\  P'\end{array}\right)=4$ would imply that $P'\hada L$ and $P\hada L'$ are disjoint, while they meet in $P\hada P'$.
\end{oss}

 \begin{lemm}\label{plane}
 Hypotheses as in  Theorem \ref{teop3}. If  there exist $P,Q\in X$ with $P\neq Q$ such that $(P\hada L' )\cap (Q\hada L')\not=\emptyset,$ then  $X\hada L'$ and  $X'\hada L$ are  contained in the same plane. Similarly if there exist $P', Q'\in X'$ with $P'\neq Q'$ such that $(P'\hada L)\cap (Q'\hada L)\not=\emptyset.$
 \end{lemm}
\begin{proof} We only prove the first statement, the other being similar.
\par
Since  $(P\hada L' )\cap (Q\hada L')\not=\emptyset,$ then  they determine a plane $\Pi$.  Now, let $P'$ be any point of $X'$ and
consider the line $P'\hada L.$ By the proof of Theorem \ref{teop3}, one has
\[
(P \hada L')\cap( P'\hada L)=P\hada P' \mbox { and } (Q \hada L')\cap( P'\hada L)=Q\hada P'
\]
hence $P\hada P',Q\hada P'$ are distinct points of $\Pi$ and thus also the line $P'\hada L$ lies in $\Pi$, hence $X'\hada L$ is contained in the plane $\Pi.$ Now let $Q'$ be any other point of $X',$ then $(P'\hada L)\cap (Q'\hada L)\not=\emptyset$ and, form what we have proved,
$ P'\hada L$ and  $Q'\hada L$ both lie in $\Pi.$
With the same reasoning we have that $X\hada L'$ is contained in the plane determined by $ P'\hada L$ and  $Q'\hada L,$ which is $\Pi.$
\end{proof}

 \begin{cor}\label{generic3}
Let $L, L'$ be two generic distinct lines in $\mathbb{P}^3.$  There is  a generic choice of a finite set of points  $X\subseteq L$   for which it is possible a generic choice of a  finite set of points $X'\subseteq L'$   such  that:
\begin{enumerate}
\item   $X\hada X'=(X\hada L')\cap(X'\hada L)$ and $|X\hada X'|=|X||X'|.$
\item  $L \hada L'$ is an   irreducible and non-degenerate  quadric, and   $X\hada L'$ and $X'\hada L$ are lines of the two different rulings.
 \end{enumerate}
 \end{cor}

\begin{proof}
\begin{enumerate}
\item  [  $ (1)$]
We may assume that     $L\cap\Delta_1=\emptyset=L'\cap\Delta_1, $ so that $L\cap \Delta_2$ and  $L'\cap \Delta_2$ are finite. By Lemma \ref{genericH} we can write   $L=H\cap K$ and  $L'=H' \cap K',$  with  $ A, B, A', B'\not\in\Delta_2,$ where     $A$, $B$, $A'$ and $B'$  are the points which correspond to $H,K,H',K'$ in the dual space.
\par If $rank\left(\begin{array}{c}A\\  B\\  A'\\  B'\end{array}\right)\hada\left(\begin{array}{c}P\\  P\\  P'\\ P'\end{array}\right)=2,$ for all $P\in L\setminus\Delta_2$ and $P'\in L'\setminus\Delta_2,$ then $P\hada L'=P'\hada L$ is a line, say $L'',$ for all  $P\in L\setminus\Delta_2$ and $P'\in L'\setminus\Delta_2.$  Thus $$L\hada L'=\overline{\bigcup\limits_{P\in L}\{P\hada L'\}}=\overline {(\bigcup\limits
  _{P\in L\setminus\Delta_2}\{P\hada L'\})\bigcup (\bigcup\limits_{P\in L \cap \Delta_2}\{P\hada L'\})}=$$
  $$= L''\bigcup\left(\bigcup\limits_{P\in L \cap \Delta_2}\{P\hada L'\}\right),$$ which is a union of a line and a finite number of linear spaces of dimension less than or equal to $1.$  This  contradicts  Theorem 6.8 of \cite{BCK} in view of Remark 6.9 of \cite{BCK}.
\par Hence there exist $P\in L\setminus\Delta_2$ and $P'\in L'\setminus\Delta_2$ such that
$$rank\left(\begin{array}{c}A\\  B\\  A'\\  B'\end{array}\right)\hada\left(\begin{array}{c}P\\  P\\  P'\\ P'\end{array}\right)=3.$$
  Consider  a point $Q\in L$ and   the $4\times 4$ matrix
 $$M(\lambda,\mu)=\left(\begin{array}{c}A\\  B\\  A'\\  B'\end{array}\right)\hada\left(\begin{array}{c}\lambda P+\mu Q\\ \lambda  P+\mu Q\\  P' \\ P' \end{array}\right)$$
  with $[\lambda: \mu]\in \mathbb{P}^1.$ Now we get the conclusion by mimicking the proof  of Corollary \ref{generic}  and by applying Theorem \ref{teop3}.
\item  [ $ (2)$]
 Since $L$ and $L'$ are generic, by Theorem 6.8 of \cite{BCK}, $L \hada L'$ is a quadric, in fact an irreducible one, as noticed right after Remark 2.5 of \cite{BCK}.   Since the quadric is irreducible, then
\[
(P'\hada L) \cap (Q'\hada L)=\emptyset \quad \forall P',Q' \in X'
\]
and similarly
\[
(P\hada L') \cap (Q\hada L')=\emptyset \quad \forall P,Q \in X.
\]
In fact, suppose $P'\hada L$ and $Q'\hada L$ intersect in a point, then, by Lemma \ref{plane},   the  lines $P'\hada L,$ $Q'\hada L$ and $P\hada L'$ are all distinct and lie in the same plane.
 But then  $L\hada L'$ would be reducible.
 \par On the other hand $P\hada L'$ and $P'\hada L$ intersect in $P\hada P'$ for all $P\in X$ and for all $P'\in X'.$
 \par Therefore $L\hada L'$ is also non degenerate.
 \end{enumerate}
 \end{proof}

 \begin{oss} \label{unique}\rm
  If both $|X|$ and $|X'|$ are strictly greater than $2,$ then we have at least three skew lines each with at least three points of $X\hada X'$ and this is enough to prove that $L\hada L'$ is the unique quadric through $X\hada X'.$ It would be interesting to understand the geometry of $X\hada X'$ on such a quadric.
 \end{oss}

 \begin{ex}
\rm In this example we compute the ideal of $X\hada X'$ and its Hilbert function, where $X$ and $X'$ are two sets of collinear points satisfying the hypotheses of Corollary \ref{generic3}.
\par Let $H,K$ be the planes defined by $x_0-x_1+x_2+2x_3=0$ and $x_0+2x_1-x_2+x_3=0$ and let $H',K'$ be the planes defined by $x_0+2x_1-2x_2+x_3=0$ and $2x_0+2x_1+x_2-4x_3=0.$  Let $L=H\cap K$ and $L'=H'\cap K'.$ Choose $X\subset L$ and $X'\subset L'$ where

\[X=\{[-2:1:1:1], [-1:-1:-2:1], [-2:3:4:1] \}
\]
and
\[
 X'=\{[-1:2:2:1], [11:-8:-2:1], [-7:7:4:1] \}.
\]
By computing the ideal of $X\hada X'$ with \texttt{Singular}, we obtain
\par
\noindent
{\small{\begin{verbatim}
ideal I=HPr(X,X',3)

I[1]=-3/31x(0)^2+41/62x(0)x(1)-15/31x(1)^2-169/186x(0)x(2)+
       +59/62x(1)x(2)-21/62x(2)^2-59/62x(0)x(3)+
        +845/186x(1)x(3)-287/124x(2)x(3)-105/62x(3)^2
I[2]=-3/31x(0)x(1)^2+18/31x(1)^3+10/31x(0)x(1)x(2)-
       -87/31x(1)^2x(2)-25/93x(0)x(2)^2+140/31x(1)x(2)^2-
       -75/31x(2)^3-24/31x(0)x(1)x(3)+137/31x(1)^2x(3)+
       +35/31x(0)x(2)x(3)-2261/186x(1)x(2)x(3)+
       +505/62x(2)^2x(3)+51/31x(0)x(3)^2-362/31x(1)x(3)^2+
       +1443/62x(2)x(3)^2--873/31x(3)^3
I[3]=6/31x(0)x(1)x(2)-36/31x(1)^2x(2)-10/31x(0)x(2)^2+
       +114/31x(1)x(2)^2-90/31x(2)^3+42/31x(0)x(2)x(3)-
       -238/31x(1)x(2)x(3)+303/31x(2)^2x(3)-1152/155x(0)x(3)^2+
        +192/31x(1)x(3)^2+3762/155x(2)x(3)^2-1728/31x(3)^3
I[4]=6/31x(0)x(2)^2-36/31x(1)x(2)^2+54/31x(2)^3+
        +3456/775x(0)x(1)x(3)-576/155x(1)^2x(3)-
       -1584/155x(0)x(2)x(3)+7944/775x(1)x(2)x(3)-
        -1476/155x(2)^2x(3)+3456/775x(0)x(3)^2+
        +4608/155x(1)x(3)^2-28296/775x(2)x(3)^2+
        +5184/155x(3)^3
I[5]=-6/5x(1)^3+6x(1)^2x(2)-10x(1)x(2)^2+50/9x(2)^3+
        +336/155x(0)x(1)x(3)-1954/155x(1)^2x(3)-
        -112/31x(0)x(2)x(3)+15742/465x(1)x(2)x(3)-
        -5972/279x(2)^2x(3)-428/155x(0)x(3)^2+
        +13652/465x(1)x(3)^280218/1395x(2)x(3)^2+
        +91204/1395x(3)^3
I[6]=-x(1)^2x(2)+10/3x(1)x(2)^2-25/9x(2)^3+
        +56/31x(0)x(2)x(3)-884/93x(1)x(2)x(3)+
        +2986/279x(2)^2x(3)-3956/465x(0)x(3)^2+
       +392/31x(1)x(3)^2+28958/1395x(2)x(3)^2-
        -16252/279x(3)^3
I[7]=-x(1)x(2)^2+5/3x(2)^3+3956/775x(0)x(1)x(3)-
        -1176/155x(1)^2x(3)-5102/465x(0)x(2)x(3)+
        +15944/775x(1)x(2)x(3)-6703/465x(2)^2x(3)+
         +3956/775x(0)x(3)^2+12724/465x(1)x(3)^2-
         -88388/2325x(2)x(3)^2+16252/465x(3)^3
I[8]=-1/4x(2)^3-1278/775x(0)x(1)x(3)-252/155x(1)^2x(3)+
        +687/155x(0)x(2)x(3)+3228/775x(1)x(2)x(3)-
        -172/155x(2)^2x(3)-4518/775x(0)x(3)^2-234/155x(1)x(3)^2
        +988/775x(2)x(3)^2-422/155x(3)^3
        \end{verbatim}}}
\noindent
 whose Hilbert function is given by
\begin{verbatim}
HF(3,I,0)=1
HF(3,I,1)=4
HF(3,I,2)=9
HF(3,I,3)=9
\end{verbatim}
\noindent that is $H(t) = 9$   for $ t \geq 2$.
 \end{ex}

 \bigskip
  The example above shows that the finite set   $X\hada X'$ in Corollary \ref{generic3}, in general, is not a complete intersection. However we are able to compute its Hilbert function in the case $|X|=|X'|$ and this allows us to prove that $X\hada X'$ is never a complete intersection as long as $m>1$ (obviously it is for $m=1$).
  \bigskip
 \begin{teo}\label{HF3}
 Hypotheses as in Corollary \ref{generic3}.  Also suppose $|X|=|X'|=m,$ then $\tau_{X\hada X'}=m-1$ and  $HF_{X\hada X'}=HF_X HF_{ X'}.$
 \end{teo}
 \begin{proof}For $m=2$ the four points of $X\hada X'$ cannot be coplanar since they belong to the two skew lines of
 $X\hada L'$ and so $HF_{X\hada X'}(t) =4,$ for all $t\geq 1.$
 \par
For $m\geq 3,$ by intersection theory we have that the quadric $L\hada L'$ (which is the unique quadric through $X\hada X'$ by Remark  \ref{unique}) is a fixed component of $I_{X\hada X'}$  in each degree $2\leq t<m.$  Then, for $0\leq t<m$ we have
$$HF_{X\hada X'}(t) = \binom{t+3}{3}-\binom{t+1}{3}=(t+1)^2$$
which equals  $HF_X (t)HF_{ X'}(t)$ since $HF_X (t)=HF_{X'}(t)=t+1$ for $t<m.$
\par In particular, $HF_{X\hada X'}(m-1)=m^2= |X| |X'|,$ hence $\tau_{X\hada X'}=m-1$ and, for all $t\geq m-1,$   $HF_{X\hada X'}(t) =m^2= HF_X (t)HF_{X'}(t).$
 \end{proof}
Obviously Theorem \ref{HF3} works also for $m=1.$
 \begin{oss} \rm
If $X$ is a finite set of projective points we set
 $$h_X=(HF_X(0),\dots,HF_X(\tau_X)).$$
 With this notation we can rephrase Theorem \ref{HF3} as
 $$h_{X\hada X'}=h_X\hada h_{X'}.$$
 \end{oss}
 \bigskip
 The following example shows that we may still have $HF_{X\hada X'}=HF_X HF_{ X'}$ even when $|X|\neq |X'|$. It may be worth to investigate if this is always the case, under the hypotheses of Corollary \ref{generic3} (see Example \ref{notprod}).

\begin{ex}\rm
Let $H,K$ be the planes defined by $11x_1-14x_2-2x_3$ and $22x_0-25x_2-13x_3$ and let $H',K'$ be the planes defined by $21x_1-2x_2-11x_3$ and $7x_0-6x_2+2x_3$  Let $L=H\cap K$ and $L'=H'\cap K'.$ Choose $X\subset L$ and $X'\subset L'$ where
\[X=\{[4: 4: 3: 1], [7: 4: 2: 8], [11:8:5:9] \}
\]
and
\[ X'=\{
[2: 3: 4:  5], [6: 4: 9:  6], [18:17:30:27], [94:76:149:118 ]\}.
\]
Let $I,J,K$ be respectively the ideals of $X$, $X'$ and $X\hada X'$. By \texttt{Singular} we obtain
\begin{verbatim}
HF(3,I,0)=1      HF(3,J,0)=1      HF(3,K,0)=1
HF(3,I,1)=2      HF(3,J,1)=2      HF(3,K,1)=4
HF(3,I,2)=3      HF(3,J,2)=3      HF(3,K,2)=9
HF(3,I,3)=3      HF(3,J,3)=4      HF(3,K,3)=12
HF(3,I,4)=3      HF(3,J,4)=4      HF(3,K,4)=12
\end{verbatim}

\end{ex}

\begin{cor}\label{ci} Hypotheses as in Corollary \ref{generic3} and $|X|= |X'|=m\geq 2.$ Then $X\hada X'$ is not a complete intersection.
 \end{cor}
\begin{proof} First assume $m=2.$ Then
$ dim_{{\mathbb K}}\left(I_{X\hada X'}\right)_t=\left\{\begin{array}{cc} 0 & t=0,1\\ 6 & t=2\end{array}\right. .$ Thus a minimal system of generators of $I_{X\hada X'}$ contains at least six quadrics and so $X\hada X'$ cannot be complete intersection.
\par
Now assume $m\geq 3.$ From Remark \ref{unique} we know that $$dim_{{\mathbb K}}\left(I_{X\hada X'}\right)_t=\left\{\begin{array}{cc}0 & t=0,1\\ 1 & t=2\end{array}\right. ,$$ and so $ dim_{{\mathbb K}}\left(I_{X\hada X'}\right)_t\geq \binom{t+1}{3}, ~\forall t\geq 2.$ As in the proof of Theorem
\ref{HF3} we have that the quadric $L\hada L'$ is a fixed component of $I_{X\hada X'}$ in each degree $2\leq t<m,$ and so we need
$\binom{m+3}{3}-m^2-\binom{m+1}{3}=2m+1$ generators of degree $m.$ Thus a minimal system of generators of $I_{X\hada X'}$ consists of $2m+2>3$ forms and so $X\hada X'$ cannot be complete intersection.
\end{proof}

\bigskip
If we drop some of the assumptions of Corollary \ref{generic3} several behaviours may occur, as the following examples show.
\begin{ex}\label{notprod}
\rm
Let $H,K$ be the planes defined by $x_1-x_3$ and $14x_0-27x_2+10x_3$ and let $H',K'$ be the planes defined by $9x_1+5x_2-11x_3$ and $x_2-x_2$. Let $L=H\cap K$ and $L'=H'\cap K'.$ Note that this time $L\cap \Delta_1\not= \emptyset$ and  $L'\cap \Delta_1\not= \emptyset$. Choose $X\subset L$ and $X'\subset L'$ where
\[X=\Big\{
\begin{array}{l}
[1: 4:2:4], [8: 5: 6: 5], [37:40:34:40] , \\

[9:9:8:9], [65:98:70:98]
\end{array}\Big\}
\]
and
\[ X'=\Big\{
\begin{array}{l}[2: 5:  2:  5], [3:2: 3:3], [24:27,24, 33], \\

[13: 16: 13:19], [130:127:130:163]\end{array}\Big\}.
\]
Let $I,J,K$ be respectively the ideals of $X$, $X'$ and $X\hada X'$. By \texttt{Singular} we obtain
\begin{verbatim}
HF(3,I,0)=1      HF(3,J,0)=1      HF(3,K,0)=1
HF(3,I,1)=2      HF(3,J,1)=2      HF(3,K,1)=3
HF(3,I,2)=3      HF(3,J,2)=3      HF(3,K,2)=6
HF(3,I,3)=4      HF(3,J,3)=4      HF(3,K,3)=10
HF(3,I,4)=5      HF(3,J,4)=5      HF(3,K,4)=15
HF(3,I,5)=5      HF(3,J,5)=5      HF(3,K,4)=19
HF(3,I,6)=5      HF(3,J,6)=5      HF(3,K,6)=22
HF(3,I,7)=5      HF(3,J,7)=5      HF(3,K,7)=24
HF(3,I,8)=5      HF(3,J,8)=5      HF(3,K,8)=25
\end{verbatim}
\par
Notice that, in this case, the Hilbert function of $X\hada X'$ is not the product of the Hilbert functions of $X$ and $X'$.
\par
As a matter of fact, looking at the ideal of $X\hada X'$, we can notice that the first generator is
\begin{verbatim}
K[1]=14*x(0)-18*x(1)-27*x(2)+22*x(3)
\end{verbatim}
that is, $X\hada X'$ is a planar set of points. Moreover the first difference of its Hilbert function is $(1,2,3,4,5,4,3,2,1)$ showing that $X\hada X'$ is a complete intersection.
\end{ex}

\bigskip
The following two examples show  that $L\hada L'$ can be a quadric (necessarily  irreducible) also under the condition that $L\cap\Delta_1\neq \emptyset$ or  $L\cap\Delta_1\neq \emptyset\neq L'\cap\Delta_1.$ In both examples $X\hada X'$ is not a complete intersection.

 \begin{ex}\label{9points}
  \rm In this example we compute the ideal of $L\hada L'$ and the ideal of  $X\hada X'$ with  its Hilbert function, where $X$ and $X'$ are two sets of collinear points satisfying the hypotheses of Theorem \ref{teop3},   $L\cap\Delta_1\neq \emptyset$ and  $L'\cap\Delta_1=\emptyset.$
\par Let $H,K$ be the planes defined by $x_0+2x_1+x_2+x_3=0$ and $x_0+x_1+x_2-3x_3=0$ and let $H',K'$ be the planes defined by $x_0+2x_1-2x_2+x_3=0$ and $2x_0+2x_1+x_2-4x_3=0.$  Let $L=H\cap K$ and $L'=H'\cap K'.$ Choose $X\subset L$ and $X'\subset L'$ where
\[X=\{[4:-4:3:1], [6:-4:1:1], [5:-4:2:1] \}
\]
and
\[ X'=\{[-1:2:2:1], [11:-8:-2:1], [-7:7:4:1] \}.
\]
By computing the ideal of $L\hada L'$ and the ideal of $X\hada X'$ with \texttt{Singular}, we obtain
{\small{\begin{verbatim}

ideal J=HPr(L,L',3)
J[1]=1/5xy-21/50y^2-3/5yz-12/5xw+77/25yw-14/5zw+588/25w^2

ideal I=HPr(X,X',3)
I[1]=1/5x(0)x(1)-21/50x(1)^2-3/5x(1)x(2)-12/5x(0)x(3)+
       +77/25x(1)x(3)-14/5x(2)x(3)+588/25x(3)^2
I[2]=1/5x(0)^3-9261/5000x(1)^3-9/5x(0)^2x(2)+
        +27/5x(0)x(2)^2-27/5x(2)^3-15x(0)^2x(3)+
        +25137/625x(1)^2x(3)+27x(0)x(2)x(3)+
        +54x(2)^2x(3)+370x(0)x(3)^2-350763/1250x(1)x(3)^2-
      -165x(2)x(3)^2-1389774/625x(3)^3
I[3]=-x(0)^2x(2)+441/100x(1)^2x(2)+6x(0)x(2)^2-9x(2)^3+
         +40x(0)x(2)x(3)-1071/25x(1)x(2)x(3)+90x(2)^2x(3)-
        -15x(1)x(3)^2-14274/25x(2)x(3)^2+180x(3)^3
I[4]=-x(0)x(2)^2+21/10x(1)x(2)^2+3x(2)^3-3/10x(1)^2x(3)-
        -11/2x(1)x(2)x(3)-101/5x(2)^2x(3)+36/5x(1)x(3)^2+
        +66x(2)x(3)^2-216/5x(3)^3
I[5]=1/10x(1)^3+2/5x(1)^2x(3)-464/5x(1)x(3)^2-3584/5x(3)^3
I[6]=-x(1)^2x(2)-16x(1)x(2)x(3)+320x(0)x(3)^2-672x(1)x(3)^2-
        -224x(2)x(3)^2-3136x(3)^3
I[7]=-x(1)x(2)^2+32/5x(0)^2x(3)-3528/125x(1)^2x(3)-
        -152/5x(0)x(2)x(3)-84/5x(1)x(2)x(3)+28/5x(2)^2x(3)-
        -448x(0)x(3)^2+84672/125x(1)x(3)^2-392/5x(2)x(3)^2+
         +471968/125x(3)^3
I[8]=-25/16x(2)^3-6x(0)^2x(3)+5367/200x(1)^2x(3)+
        +57/2x(0)x(2)x(3)+181/8x(1)x(2)x(3)+6x(2)^2x(3)+
        +365x(0)x(3)^2-53229/100x(1)x(3)^2+349/4x(2)x(3)^2-
         -76294/25x(3)^3

\end{verbatim}}}
\noindent whose Hilbert function is given by
\begin{verbatim}
HF(3,I,0)=1
HF(3,I,1)=4
HF(3,I,2)=9
HF(3,I,3)=9
\end{verbatim}
\noindent that is $H(t) = 9$   for $ t \geq 2$.

 \end{ex}
 \bigskip

 \begin{ex}
 \rm In this example we compute the ideal of $L\hada L'$ and  the ideal of  $X\hada X'$ with  its Hilbert function, where $X$ and $X'$ are two sets of collinear points satisfying the hypotheses of Theorem \ref{teop3} and  $L\cap\Delta_1\neq \emptyset \neq L'\cap\Delta_1.$
 \par Let $H,K$ be the planes defined by $x_0+2x_1+x_2+x_3=0$ and $x_0+x_1+x_2-3x_3=0$ and let $H',K'$ be the planes defined by $x_0+x_1-2x_2+x_3=0$ and $x_0+x_1+x_2-4x_3=0.$  Let $L=H\cap K$ and $L'=H'\cap K'.$ Choose $X\subset L$ and $X'\subset L'$ where
$$X=\{[4:-4:3:1], [6:-4:1:1], [5:-4:2:1] \} $$
and
$$ X'=\{[1:-1:\frac{5}{3}:1],  [2:-2:\frac{5}{3}:1],  [3:-3:\frac{5}{3}:1] \}.$$
By computing the ideal of $L\hada L'$ and the ideal of $X\hada X'$ with \texttt{Singular}, we obtain
{\small{
\begin{verbatim}
ideal J=HPr(L,L',3);
J[1]=-3/5x(1)x(2)-4x(0)x(3)+7x(1)x(3)-28/5x(2)x(3)+196/3x(3)^2

ideal I=HPr(x(0),x(0)',3);
I[1]=-3/5x(1)x(2)-4x(0)x(3)+7x(1)x(3)
I[2]=-3/5x(2)^3+6x(2)^2x(3)-55/3x(2)x(3)^2+50/3x(3)^3
I[3]=-x(0)x(2)^2-5/3x(0)x(2)x(3)-50x(0)x(3)^2+250/3x(1)x(3)^2
I[4]=-x(0)^2x(2)-40/3x(0)^2x(3)+185/6x(0)x(1)x(3)-
        -25/2x(1)^2x(3)
I[5]=1/4x(1)^3-6x(1)^2x(3)+44x(1)x(3)^2-96x(3)^3
I[6]=-x(0)x(1)^2+24x(0)x(1)x(3)-176x(0)x(3)^2-288/5x(2)x(3)^2+
       +672x(3)^3
I[7]=-x(0)^2x(1)+24x(0)^2x(3)+132/5x(0)x(2)x(3)+
       +216/25x(2)^2x(3)-308x(0)x(3)^2-1008/5x(2)x(3)^2+
       +1176x(3)^3
I[8]=-x(0)^3+90x(0)^2x(3)-111x(0)x(1)x(3)+45x(1)^2x(3)+
       +99x(0)x(2)x(3)+162/5x(2)^2x(3)-341x(0)x(3)^2-
       -330x(1)x(3)^2-2448/5x(2)x(3)^2+2022x(3)^3
\end{verbatim}}}
\noindent whose Hilbert function is given by
\begin{verbatim}
HF(3,I,0)=1
HF(3,I,1)=4
HF(3,I,2)=9
HF(3,I,3)=9
\end{verbatim}
\noindent that is $H(t) = 9$   for $ t \geq 2$.
\end{ex}

 \section{Computing Hadamard products in \texttt{Singular}}\label{singular}

Given the ideals $I$ and $J$ of respectively varieties $X$ and $Y$ in ${\mathbb{P}}^n$, the computation of the ideal of $X \hada Y$ may be achieved with a saturation and elimination as follows:
\begin{itemize}
\item Work in the ring ${\mathbb{C}}[y_{10}, \ldots, y_{1n}, y_{20}, \ldots, y_{2n}, x_0, \ldots, x_n]$.
\item Form the ideal $I(y_{1i}) + J(y_{2i}) + \langle x_0 - y_{10}y_{20}, \ldots, x_n - y_{1n}y_{2n} \rangle$.
\item Saturate with respect to the product $x_0 \cdots x_n$.
\item Eliminate the $2n+2$ variables $y_{i0}, \ldots, y_{in}$.
\end{itemize}
For completeness, we show here a procedure in \texttt{Singular} which performs the previous steps and that we used to compute the examples in the paper.
{\tiny{
\begin{verbatim}
LIB "ncalg.lib";
LIB "poly.lib";
LIB "rootsmr.lib";
LIB "elim.lib";


proc HPr(ideal I1, ideal I2, int n) /* where n+1 is the number of variables */
{

   ring RH=0,(y(1..2)(0..n),x(0..n)),dp;
   int i;
   ideal T1;
   ideal T2;
   poly elle1;
   poly elle2;
   poly elle3=1;
   map f1;
   map f2;

   T1=y(1)(0);
   for (i=1; i<=n; i=i+1)
   {
      elle1=y(1)(i);
      T1=T1+elle1;
   }
   f1=r,T1;
   ideal H1=f1(I1);


   T2=y(2)(0);
   for (i=1; i<=n; i=i+1)
   {
      elle2=y(2)(i);
      T2=T2+elle2;
   }
   f2=r,T2;
   ideal H2=f2(I2);

   int j;
   ideal H=0;
   for (j=0; j<=n; j=j+1)
   {
      H=H+ideal(x(j)-y(1)(j)*y(2)(j));
      elle3=elle3*x(j);
   }

  H=H+H1+H2;
  ideal Ksat=elle3;
  ideal HH=sat(H,Ksat)[1];

  ideal HHH=elim(H,1..2*(n+1));

  setring r;
  ideal HFin=imap(RH,HHH);

  return(HFin);
}

\end{verbatim}
}}

\end{document}